\documentclass[12pt]{amsart}
\usepackage{titletoc}
\usepackage[utf8]{inputenc}
\usepackage[T1]{fontenc}
\usepackage{indentfirst}
\usepackage[margin=1in]{geometry} 
\usepackage{lipsum}
\usepackage{amsmath,amscd,amsthm,amssymb,mathrsfs,stmaryrd,color,tikz-cd,graphicx,tikz,mathtools,subfigure}
\tikzcdset{arrow style=tikz,diagrams={>=stealth}}
\usepackage{enumerate}
\usepackage[bookmarks, bookmarksdepth=1, colorlinks=true, linkcolor=blue, citecolor=red, urlcolor=blue,pagebackref]{hyperref}
\usepackage{subfiles}

\usetikzlibrary{matrix,arrows,decorations.pathmorphing,cd,topaths,calc,graphs}
\usepackage[all,cmtip]{xy}

\usepackage{relsize}

\newcommand{\la}{\lambda}
\newcommand{\e}{\varepsilon}

\renewcommand{\Im}{\operatorname{Im}}

\makeatletter
\newcommand{\subalign}[1]{%
  \vcenter{%
    \Let@ \restore@math@cr \default@tag
    \baselineskip\fontdimen10 \scriptfont\tw@
    \advance\baselineskip\fontdimen12 \scriptfont\tw@
    \lineskip\thr@@\fontdimen8 \scriptfont\thr@@
    \lineskiplimit\lineskip
    \ialign{\hfil$\m@th\scriptstyle##$&$\m@th\scriptstyle{}##$\crcr
      #1\crcr
    }%
  }
}
\makeatother

\newcommand{\BC}{\mathbb{C}}

\newcommand{\BH}{\mathbb{H}}

\newcommand{\BQ}{\mathbb{Q}}
\newcommand{\BR}{\mathbb{R}}

\newcommand{\BZ}{\mathbb{Z}}

\newcommand{\sH}{\mathscr{H}}

\newcommand{\cS}{\mathcal{S}}

\newcommand{\cT}{\mathcal{T}}

\newcommand{\fg}{\mathfrak{g}}

\newcommand{\vertiii}[1]{{\left\vert\kern-0.25ex\left\vert\kern-0.25ex\left\vert #1 
    \right\vert\kern-0.25ex\right\vert\kern-0.25ex\right\vert}}

\newcommand{\rom}[1]{\uppercase\expandafter{\romannumeral #1\relax}}

\DeclareMathOperator{\supp}{Supp}

\newcommand{\GL}{\operatorname{GL}}
\newcommand{\SL}{\operatorname{SL}}

\makeatletter
\newcommand{\colim@}[2]{%
  \vtop{\m@th\ialign{##\cr
    \hfil$#1\operator@font lim$\hfil\cr
    \noalign{\nointerlineskip\kern1.5\ex@}#2\cr
    \noalign{\nointerlineskip\kern-\ex@}\cr}}%
}
\newcommand{\Lim}{%
  \mathop{\mathpalette\varlim@{\leftarrowfill@\scriptscriptstyle}}\nmlimits@
}

\newcommand{\colim}{%
  \mathop{\mathpalette\varlim@{\rightarrowfill@\scriptscriptstyle}}\nmlimits@
}

\makeatother

\newtheorem{theorem}{Theorem}[section]

\newtheorem{proposition}[theorem]{Proposition}

\newtheorem{lemma}[theorem]{Lemma}

\theoremstyle{definition}

\theoremstyle{remark}

\newtheorem*{rmk}{Remark}

\usepackage[new]{old-arrows}

\usepackage{amssymb}

\usepackage[utf8]{inputenc}

\title{Weighted geodesic restrictions of arithmetic eigenfunctions}
\author{Jiaqi Hou and Xiaoqi Huang}
\address{Department of Mathematics, Louisiana State University, Baton Rouge, LA 70803, USA}
\email{jhou7@lsu.edu}
\email{xhuang49@lsu.edu}
\date{}

\usepackage{hyperref}
\numberwithin{equation}{section}

\begin{document}

\begin{abstract}
    Let $X$ be an arithmetic hyperbolic surface, $\psi$ a Hecke-Maass form, $\ell$ a geodesic segment on $X$, and $\mu$ a Borel measure supported on $\ell$ with dimension greater than 1/2. We obtain a power saving over the local bound of Eswarathasan and Pramanik \cite{EP22} for the $L^2$ norm of $\psi$ with respect to $\mu$, which is a weighted generalization of Marshall's geodesic restriction bound \cite{Mar16} and is proved by applying the method of arithmetic amplification. 
    
    On a general 2-dimensional Riemannian manifold, we also obtain a Kakeya-Nikodym bound for the $L^2$ norm of any Laplace-Beltrami eigenfunction with respect to a Borel measure supported on a geodesic segment with dimension greater than 1/2.
\end{abstract}

\maketitle

\section{Introduction}

The purpose of this paper is to obtain improved $L^2$ eigenfunction estimate restricted to a general Borel sets on an arithmetic hyperbolic surface.

Let $(M,g)$ be a  compact $n$-dimensional Riemannian manifold, 
$\Delta_g$ be the
Laplace-Beltrami operator associated with the metric $g$ on $M$. We let $0=\la_0^2<\la^2_1\le \la^2_2\le \cdots$ denote the
eigenvalues labeled with respect to the multiplicity of 
${-\Delta_g}$
and $e_{\la_j}$ the associated $L^2$-normalized eigenfunctions, that is,
\begin{equation}\label{1.2}
(\Delta_g+\la^2_j )e_{\la_j}=0, \qquad
\text{and } \quad \int_M |e_{\la_j}(x)|^2 \, dx=1.
\end{equation}

There are several ways to quantify the concentration and size of eigenfunctions $e_\la$ with $\la=\lambda_j>0$. A universal result of  Sogge~\cite{sogge881} states that on an 
$n$-dimensional compact manifold without boundary, for any $2\le q \le\infty$,
\begin{equation}\label{r1}
\|e_\la\|_{L^q(M)}\lesssim \la^{\sigma(q)},\qquad\sigma(q)=\max \{\tfrac{n-1}{2}(\tfrac12-\tfrac1q), \tfrac{n-1}{2}-\tfrac nq\}.
\end{equation}
In \cite{burq2007restrictions}, Burq-G\'erard-Tzvetkov studied restrictions of eigenfunctions to submanifolds. In particular, if $\ell$ is a smooth curve on a Riemannian surface, they showed that 
\begin{equation}\label{r2}
\|e_\la\|_{L^q(\ell)}\lesssim \la^{\sigma_1(q)},\qquad\sigma_1(q)=\max \{\tfrac14, \tfrac{1}{2}-\tfrac 1q\}.
\end{equation}
Both \eqref{r1} and \eqref{r2} are sharp on the standard round sphere by taking $e_\la$ to be the zonal functions or the highest weight spherical harmonics. Under additional geometric assumptions--such as nonpositive sectional curvature or on flat tori--one can improve \eqref{r2}; see, for example,
\cite{chen2014few, bourgain2012restriction, blair2018logarithmic, xi2017improved} and the references therein for more details.

Eswarathasan and Pramanik \cite{EP22} then considered the restriction of eigenfunctions to arbitrary Borel sets. To describe the results in this direction, we recall the Frostman Lemma in $\mathbb{R}^n$ (see e.g., \cite{Mat15}).  It states that for a Borel set $B\subset \mathbb{R}^n$ and $s>0$, the $s$-dimensional Hausdorff measure satisfies $\mathcal{H}^s(B)>0$ if and only if there exists a nontrivial finite Borel measure $\mu$ supported on $B$
and a constant $C>0$ such that 
\[
    \mu(B(x,r)) \le C r^s, \qquad \forall x \in \mathbb{R}^n,\ r>0.
\]
In this paper, we say that a Borel measure $\mu$ on $M$ has dimension $\alpha$ if there exists a constant $C>0$ such that
\[
    \mu(B(x,r)) \le C r^\alpha, \qquad \forall x \in M,\ 0<r<\text{inj}(M).
\]
Here, $B(x,r)$ denotes the geodesic ball of radius $r$ centered at $x$, and $\text{inj}(M)$ denotes the
injectivity radius of $M$.

For an $\alpha$-dimensional Borel measure $\mu$ on a 
Riemannian surface, analogues of \eqref{r1} on the size of $\|e_\la\|_{L^q(M;\, d\mu)}$ was obtained in \cite{EP22}.
More recently, Gao-Miao-Xi \cite{GMX24} proved sharp estimate on $\|e_\la\|_{L^q(M;\, d\mu)}$ for all $0<\alpha \le 2$ and $q\ge 2$ on any Riemannian surface.  In the special case $q=2$, their conclusions can be stated as follows:
\begin{theorem}[\cite{EP22, GMX24}]
Let $\alpha\in (0,2]$, $\mu$ be an $\alpha$-dimensional probability measure on a Riemannian surface $M$. Then for any $\e>0$, there exist $C_\e>0$ such that  
\begin{equation}\label{r3}
\|e_\la\|_{L^2(M;\, d\mu)}\lesssim \la^{\gamma(\alpha)+\e}\|e_\la\|_{L^2(M)},
\end{equation}
where $$ \gamma(\alpha)=\begin{cases}
    \tfrac12-\frac \alpha 2, \qquad &\text{if }\,0<\alpha\le \frac12, \\
    \frac 14,&\text{if }\,\frac12 < \alpha\le 1,\\
    \frac{2-\alpha}{4},&\text{if }\,1<\alpha\le 2.
\end{cases}
$$
Moreover, \eqref{r3} is sharp on $S^2$ up to the $\la^\e$ loss. 
\end{theorem}

\subsection{Improvements in the arithmetic setting}
Now we take the surface to be $X=\Gamma\backslash\BH$, an arithmetic hyperbolic surface, where $\Gamma$ is a cocompact congruence lattice in $\SL(2,\BR)$ arising from a quaternion division algebra over $\BQ$ (see Section \ref{sec: arithmetic surface and Maass forms} for the precise definitions). Let $\psi$ be an $L^2$-normalized Hecke–Maass form on $X$ with spectral parameter $\lambda$. Namely, $\psi$ is a joint eigenfunction of the Laplace-Beltrami operator $\Delta$ and all unramified Hecke operators on $X$, and satisfies $\Delta \psi+(1/4+\lambda^2)\psi=0$. As we are considering large-eigenvalue asymptotics, we will also assume that $\lambda>0$ to exclude exceptional Laplace eigenvalues. For these $X$, Iwaniec and Sarnak \cite{IS95} showed that the classical local bound $\| \psi\|_\infty \lesssim\lambda^{1/2}$ (taking $n=2$ and $q=\infty$ in \eqref{r1}) can be improved a power to $\|\psi\|_\infty \lesssim_\epsilon\lambda^{5/12+\epsilon}$. Therefore, the trivial estimate
\[
\| \psi\|_{L^2(\mu)} \leq\| \psi\|_\infty\left(\int_X d\mu(x)\right)^{1/2}\lesssim_\epsilon\lambda^{5/12+\epsilon}
\]
provides a power saving over \eqref{r3} when $\alpha<\frac{1}{6}$. The main result of this paper is an improvement of \eqref{r3} for Hecke-Maass forms on $X$ in the range $\alpha\in(\frac{1}{2},1]$ with $\mu$ supported on a geodesic segment.

Let $M$ be any Riemannian surface and $\ell(s): [0,1]\to M$ be a geodesic segment of unit length, we define $\mu$ to be a Borel measure of dimension $\alpha\in (0,1]$ supported on $\ell$ if $\mu$ is the push-forward of some $\alpha$-dimensional Borel  measure $\nu$ on $[0,1]$ under $\ell$, i.e., 
\begin{equation}\label{r4}
    \mu(E)=\nu(\ell^{-1}(E)) \,\,\,\text{for all Borel sets }\,\,E\subset M.
\end{equation}

\begin{theorem}\label{thm: main}
  Let $\psi$ be an $L^2$-normalized Hecke–Maass form on $X$ with spectral parameter $\lambda$. For any geodesic segment $\ell$ of unit length, and any Borel measure $\mu$ supported on $\ell$, with dimension $\alpha\in (\frac12, 1]$, then we have
    \begin{equation}\label{art}
        \| \psi \|_{L^2(\mu)} = \left(\int_0^1 |\psi(\ell(s))|^2d\nu(s) \right)^{1/2}\lesssim_{\alpha,\epsilon}\lambda^{\frac{1}{4}-\delta(\alpha)+\epsilon}
    \end{equation}
    where the measure $d\nu$ is defined in \eqref{r4} and 
    \begin{align*}
        \delta(\alpha)=\frac{\alpha-\frac{1}{2}}{24(\alpha-\frac{1}{2})+2}.
    \end{align*}
The implied constant is independent of $\ell$.
\end{theorem}

In the special case $\alpha=1$, Theorem~\ref{thm: main} is due to Marshall \cite{Mar16}. 
The proof of Theorem~\ref{thm: main} follows by reducing the restriction norm to a weighted geodesic restriction norm and adapting the arguments of Marshall~\cite{Mar16}, which are based on the arithmetic amplification method, to this weighted setting.

\subsection{A Kakeya-Nikodym bound}
Following ideas in the work of Bourgain~\cite{bourgain2009geodesic}, 
we also obtain the following result on a general compact manifold.
\begin{theorem}\label{thm: main1}
    Let $(M,g)$ be a compact 2-dimensional Riemannian manifold, let $\Delta_g$ be the
Laplace-Beltrami operator associated with the metric $g$ on $M$. Assume $e_\la$ is an $L^2$ normalized eigenfunction satisfying \eqref{1.2}.
For any geodesic segment $\ell$ of unit length, and any Borel measure $\mu$ supported on $\ell$, with dimension $\alpha\in (\frac12, 1)$,  we have
    \begin{equation}\label{mic}
        \| e_\la \|_{L^2(\mu)} = \left(\int_0^1 |e_\la(\ell(s))|^2d\nu(x) \right)^{1/2}\lesssim \lambda^{\frac14}\left(\sup_{\ell\in {\mathit\Pi}}\int_{\mathcal{T}_{\lambda^{-1/2}} (\ell)}|e_\la(x)|^2 dx\right)^{\alpha-\frac12}
    \end{equation}
The estimate \eqref{mic} also holds for $\alpha=1$, up to an additional logarithmic loss in $\la$.
\end{theorem}
Here $\mathit{\Pi}$ denotes the set of unit-length geodesics in $M$ and $\mathcal{T}_{\lambda^{-1/2}} (\ell)$ denotes the $\la^{-1/2}$ tubular neighborhood of the geodesic $\ell$. The quantity on the right side of \eqref{mic} is often referred to as the Kakeya-Nikodym norm. \eqref{mic} suggests that any improvement on the Kakeya-Nikodym norm bound yields a corresponding improvement in the restriction bounds of eigenfunctions to Borel measures $\mu$ satisfying dimensional assumptions. For example, on Riemannian surfaces of nonpositive curvature, one obtains improved Kakeya–Nikodym estimates over the uniform bound; see, for instance, \cite[Theorem 1.1]{blair2018concerning} and \cite[Theorem 1.2]{qs}. 
If we take $M$ to be an arithmetic hyperbolic surface $X$ as in Theorem \ref{thm: main} and $e_\la=\psi$ to be an $L^2$-normalized Hecke-Maass form with spectral parameter $\la$, then Marshall's geodesic restriction estimate \cite[Theorem 1.1]{Mar16} directly implies that
\[
\|\psi\|_{L^2(\mu)}\lesssim_{\alpha,\epsilon}\lambda^{\frac{1}{4}-\frac{1}{14}(\alpha-\frac{1}{2})+\epsilon}
\]
for any $\mu$ satisfying the assumptions in Theorems \ref{thm: main} and \ref{thm: main1}. However, it may be seen that the above bound is weaker than \eqref{art} in Theorem \ref{thm: main}.

Note that the exponents in both \eqref{art} and \eqref{mic} tend to $0$ as $\alpha$ approaches $\frac12$. Thus
 $\alpha=\frac12$ is a critical value for eigenfunction restriction problems. Actually, the functions that saturate \eqref{r3} when $\alpha=\frac12$ include both the Zoll spherical functions and the highest weight spherical harmonics. Hence, in this case, the restriction norm is sensitive to both pointwise and geodesic concentration of eigenfunctions.
 We wish to explore this critical case in future work.

The paper is organized as follows. In Section \ref{section 2}, we reduce the restriction norm to a weighted norm along geodesics using the locally constant property of eigenfunctions and prove several auxiliary estimates for the corresponding weights. In Section \ref{section: arithmetic backgroup}, we introduce the definition of Maass forms on arithmetic hyperbolic surfaces, the Hecke operators, and the related amplification inequalities.
In Sections \ref{section: geometric side estimates} and \ref{section: proving thm by amplification}, we prove Theorem~\ref{thm: main}. Finally, in Section \ref{section: proving KN bound}, we give the proof of Theorem~\ref{thm: main1}.

\subsection{Notation}
Throughout the paper, the notation $A\lesssim B$ will mean that there is a positive constant $C$ such that $|A| \leq C B$, and $A\sim B$ will mean that there are positive constants $C_1$ and $C_2$ such that $C_1B \leq A \leq C_2B$. We also use $A =O(B)$ to mean $A\lesssim B$. 

If $f\in L^1(\BR)$, the Fourier transform $\widehat{f}$ of $f$ in this paper is defined as
\begin{align*}
    \widehat{f}(\xi)=\int_{-\infty}^\infty f(x)e^{-ix\xi} dx. 
\end{align*}
If $f\in\cS(\BR)$ is a Schwartz function, the Fourier inversion formula reads
\[
    f(x) = \frac{1}{2\pi}\int_{-\infty}^\infty \widehat{f}(\xi)e^{ix\xi} d\xi.
\]

\subsection{Acknowledgements}
The authors would like to thank Yakun Xi for suggesting studying fractal restriction estimates, which motivated this project.
 The second author was supported in part by the Simons Foundation and NSF (DMS-2452860).

\section{Weight functions and auxiliary estimates}\label{section 2}

By rescaling the metric, or in the hyperbolic surface case by passing to a finite-index sublattice if necessary, we may assume throughout that the injectivity radius of our manifold is at least 10. 
We will first show the following reduction.
\begin{lemma}\label{locallemmaa}
  Let $M$ be a Riemannian surface and $\mu$ be an $\alpha$-dimensional Borel probablity measure supported on a unit-length geodesic $\ell(s): [0, 1]\to M$. Suppose that $e_\la$ is an $L^2$-normalized Laplace eigenfunction satisfying \eqref{1.2}. Then there exists a constant $C>0$ and a smooth weight function $w(s)$ supported in $|s|\le 2$ such that for any $N\geq1$
  \begin{equation}\label{2.1}
      \|e_\la\|_{L^2(\mu)}\le C\|e_\la\|_{L^2(w)} +O_N(\lambda^{-N}),
  \end{equation}
where 
$$\|e_\la\|_{L^2(w)}=\int |e_\la(\ell(s))|^2 w(s)ds.$$
And for any interval $[a-r,a+r]\subset \BR$, we have 
\begin{align}\label{eq: Frostman’s lemma}
    \int_{a-r}^{a+r} w(s)ds \lesssim r^\alpha.
\end{align}
\end{lemma}
Note that the Hecke–Maass form $\psi$ with spectral parameter $\lambda$ in Theorem~\ref{thm: main} is also an eigenfunction of the Laplacian with eigenvalue $\frac14+\la^2$ with $\la>0$ (see \eqref{def} below). Therefore, the conclusion of Lemma \ref{locallemmaa} also applies to $\psi$.

To prove Lemma~\ref{locallemmaa}, we shall use the following lemma, see e.g., \cite[Chapter 5]{sogge2017fourier} for its proof.

\begin{lemma}
    Let $(M,g)$ be a Riemannian surface with the Riemannian distance function $d_g(x,y)$ for $x,y\in M$ and the Laplace-Beltrami operator $\Delta_g$. Fix $\chi \in \mathcal S (\mathbb R)$ with $\chi(0)=1$ and $\text{supp}\, {\widehat\chi} \subset (\frac12, 1)$.  Then for any $f\in C^\infty(M)$
\begin{equation}\label{a.1}
    \chi(\sqrt{-\Delta_g}-\lambda)f(x)=
\lambda^{1/2}\int_Me^{i\lambda d_g(x,y)}a(x,y,\lambda) f(y)dy +R_\lambda f(x)
\end{equation}
where 
$$ \|R_\lambda f\|_{L^\infty(M)}\le C_N\lambda^{-N} \|f\|_{L^1(M)},\qquad\forall N=1,2,3...,
$$
$\alpha(x,y,\lambda)$ is smooth with the uniform bounds
$$ \partial_{x,y}^\gamma a(x,y,\lambda)\le C_{\gamma}, \qquad\text{ for all multi-indices } \gamma,
$$
and 
\begin{equation}\label{dist}
   a(x,y,\lambda)=0 \qquad\text{if }\, d_g(x,y)\notin (\tfrac12,1) 
\end{equation}
\end{lemma}
Write 
$$T_\la f(x) = \lambda^{1/2}\int_Me^{i\lambda d_g(x,y)}a(x,y,\lambda) f(y)dy.$$ 
Using \eqref{a.1}, we have 
\begin{equation}\label{a.3}
\begin{aligned}
      e_\la= \chi(\sqrt{\Delta_g}-\lambda)e_\la=T_\lambda e_\la +O(\lambda^{-N}) 
\end{aligned}
\end{equation}
Since $d_g(x,y)\in (\frac12, 1)$ when $(x,y)$ is in the support of the amplitute $a$, it is not hard to check
\begin{equation}\label{c}
    \partial_{s} d_g(\ell(s),y)\le C_\ell
\end{equation}
for some constant $C_\ell$ depending on the metric.
Hence if we fix a nonnegative $ \rho(s)\in C_c^\infty(\BR)$ satisfying $\rho(s)\ge 0$,  $ \rho(s)=1$ when $|s|\le \frac32$ and $\rho(s)=0$ when $|s|\ge 2$, and fix 
 $\eta$ such that $\widehat \eta \in C_c^\infty$,  $\widehat\eta(\xi)= 1$ when $|\xi|\le 2C_\ell$ and $\widehat\eta(\xi)= 0$ when $|\xi|\ge 4C_\ell$, we have for $|s|\le 1$, 
\begin{equation}\label{2.5}
\begin{aligned}
      T_\lambda e_\la (\ell(s))&=\frac{1}{2\pi}\iint_{-\infty}^\infty e^{i(s-t)\xi}  \rho(t)T_\lambda e_\la (\ell(t)) d\xi dt \\
      &=\frac{1}{2\pi}\iint_{-\infty}^\infty e^{i(s-t)\xi}\widehat \eta (\xi/\la) \rho(t)T_\lambda e_\la (\ell(t)) d\xi dt\\
      &\qquad\qquad+ \frac{1}{2\pi}\iint_{-\infty}^\infty e^{i(s-t)\xi} (1-\widehat\eta (\xi/\la)) \rho(t)T_\lambda e_\la (\ell(t)) d\xi dt. 
\end{aligned}
\end{equation}
The second term on the right side of \eqref{2.5} is $O_N(\la^{-N})$ for any $N$ if we expand $T_\lambda e_\lambda(\ell(t))$ and integrate by parts in $dt$ by applying \eqref{c}. Thus, we have 
\begin{equation}
\begin{aligned}
      T_\lambda e_\la (\ell(s))=\int T_\lambda e_\la (\ell(s))\rho(t) \lambda\eta(\lambda(s-t))dt +O(\lambda^{-N}),\,\,\,\text{if}\,\,\,|s|\le 1. 
\end{aligned}
\end{equation}

Therefore, since the measure $\mu$ is supported on $\ell$, by H\"older's inequality
\begin{align}
      & \int_M  |T_\lambda e_\la (x)|^2  d\mu (x)
       = \int_0^1  |T_\lambda e_\la (\ell(s))|^2  d\nu (s)\notag\\
       \le& \int_0^1 \left(\int| \lambda\eta(\lambda(s-t))|\rho(t) dt \right)\left(\int |T_\lambda e_\la (\ell(t))|^2 
        |\lambda\eta(\lambda(s-t))|\rho(t) dt \right)d\nu(s)+O(\lambda^{-N})\notag\\
        \lesssim &\int_0^1 \int |T_\lambda e_\la (\ell(t))|^2 
        |\lambda\eta(\lambda(s-t))|\rho(t) dt d\nu(s)+O(\lambda^{-N})\label{a.4}
\end{align}
If we use \eqref{a.3} and \eqref{a.4}, we have  
\begin{equation}\label{a.5}
\begin{aligned}
   \int  |e_\la (x)|^2  d\mu (x)& \lesssim \int_0^1\int |e_\la (\ell(t))|^2 |\lambda\eta(\lambda(s-t))|\rho(t) dt d\nu(s)+O(\lambda^{-N})\\
   &\lesssim\int_0^1\int  |e_\la (\ell(t))|^2 \sqrt{\lambda^2\eta^2(\lambda(s-t))+1} \rho(t) dt d\nu(s)+O(\lambda^{-N}).
\end{aligned}
\end{equation}

If we define the weight function $w(t)$ as 
\begin{equation}
    w(t)=  \int_0^1 \sqrt{\lambda^2\eta^2(\lambda(s-t))+1}\,\rho(t) \,d\nu(s)
\end{equation}
Then \eqref{2.1} follows from \eqref{a.5}. Now let's verify  \eqref{eq: Frostman’s lemma}, note that since $\nu$ is an $\alpha$-dimensional Borel measure on $[0,1]$, we have 
\begin{equation}\nonumber
    w(t)\lesssim   \int_0^1 (|\lambda\eta(\lambda(s-t))|+1)\,\rho(t) \,d\nu(s)\lesssim_\nu \int_0^1 |\lambda\eta(\lambda(s-t))|\,\rho(t) \,d\nu(s) +\rho(t)\
\end{equation}
Thus, it suffices to verify \eqref{eq: Frostman’s lemma} with 
$w(t)$ replaced by $$\rho(t)\int_0^1 |\lambda\eta(\lambda(s-t))|\, \,d\nu(s).$$

If $r\le \la^{-1}$, let $I_k=[a+k\la^{-1}, a+(k+1)\la^{-1}]$ for $k\in\BZ$. Since $\eta$ is a Schwartz function, it is not hard to check that for any $t\in [a-r, a+r]$ and $s\in I_k$, we have 
$ |\eta(\lambda(s-t))|\lesssim (1+ |k|)^{-2}
$. Hence,
\begin{equation}
    \begin{aligned}
       \int_{a-r}^{a+r}& \rho(t)\int_0^1 |\lambda\eta(\lambda(s-t))|\, \,d\nu(s) dt \\ & \lesssim \sum_k \int_{a-r}^{a+r} \rho(t)\int_{I_k} \la(1+ |k|)^{-2}\, \,d\nu(s) dt
    \\       &\lesssim r\la^{1-\alpha} \lesssim r^{\alpha},
    \end{aligned}
\end{equation}
where in the second inequality we used the fact that $\nu$ is an $\alpha$-dimensional Borel measure on $[0,1]$.

If $r\ge \la^{-1}$, 
let $I_k=[a+kr, a+(k+1)r]$ for $k\in \mathbb{Z}$. Then one can similarly check that for any  $t\in [a-r, a+r]$ and $s\in I_k$, we have 
$$ |\sqrt\eta(\lambda(s-t))|\lesssim (1+\la |k|r)^{-2}\lesssim (1+|k|)^{-2}.
$$
Hence, 
\begin{equation}
    \begin{aligned}
       \int_{a-r}^{a+r}& \rho(t)\int_0^1 |\lambda\eta(\lambda(s-t))|\, \,d\nu(s) dt \\ & \lesssim \sum_k \left(
    \sup_{s\in I_k}\int_{a-r}^{a+r} \rho(t)\la \sqrt{\eta}(\lambda(s-t)) dt\right)\left(\int_{I_k} (1+ |k|)^{-2}\, \,d\nu(s) \right)
    \\ &\lesssim r^{\alpha}.
    \end{aligned}
\end{equation}




\subsection{Energy-integral estimates}
Now we fix the weight function $w$ produced by Lemma \ref{locallemmaa}.
Following \cite[\S2.5]{Mat15}, for any $0<s<\alpha$, we define the $s$-energy of the measure $w(x)dx$ by
\begin{align*}
    I_s(w):=\iint_\BR \frac{w(x)w(y)}{|x-y|^s} dxdy.
\end{align*}
\cite[Theorem 2.8]{Mat15} shows that if we have \eqref{eq: Frostman’s lemma}, then $I_s(w) = O_{\alpha,s}(1)$. We first refine the above estimate as follows. We let
\begin{align*}
    I_s(\phi w)=\iint_\BR \frac{\phi(x)\overline{\phi(y)}w(x)w(y)}{|x-y|^s} dxdy.
\end{align*}

\begin{lemma}\label{lem: energy estimate}
    Suppose that $0<s<\alpha$ and $\phi\in \cS(\BR)\subset L^2(w) = L^2(\BR,w(x)dx)$. Then
    \begin{align*}
         I_s(\phi w)\lesssim_{\alpha,s}  \|\phi\|^2_{L^2(w)}
    \end{align*}
\end{lemma}
\begin{proof}
    By the mean inequality,
    \begin{align*}
        I_s(\phi w)& \le \frac{1}{2} \iint_\BR \frac{|\phi(x)|^2+|\phi(y)|^2}{|x-y|^s} w(x)w(y)dxdy\lesssim\int \left(\int \frac{w(y)dy}{|x-y|^s}\right) |\phi(x)|^2 w(x)dx.
    \end{align*}
    It suffices to show that the inner integral is $O_{\alpha,s}(1)$. As both $w(x)$ and $w(y)$ are nonzero only when $|x|,|y|\leq2$, this follows from the next lemma with $\delta=4$.
\end{proof}

\begin{lemma}\label{lem: one variable energy estimate}
    If $0<s<\alpha$, $0<\delta\leq 100$ and $x\in \BR$, then
    \begin{equation}\label{db}
        \int_{-\infty}^\infty 1_{|x-y|\leq\delta} \frac{w(y)}{|x-y|^s} dy\lesssim_{\alpha,s} \delta^{\alpha-s}.
    \end{equation}
\end{lemma}

\begin{proof}
    We write $|x-y|^{-s} = s \int_{|x-y|}^\infty r^{-s-1} dr$, so
    \begin{align*}
         \int_{-\infty}^\infty 1_{|x-y|\leq\delta} \frac{w(y)}{|x-y|^s}dy&= s \int_{-\infty}^\infty \left(\int_{|x-y|}^\infty r^{-s-1} dr \right)1_{|x-y|\leq\delta} w(y)dy\\
         &=s \int_{-\infty}^\infty \left(\int_{|y|}^\infty r^{-s-1} dr \right)1_{|y|\leq\delta} w(y+x)dy\\
         &= s \int_{-\delta}^\delta \left(\int_{|y|}^\infty r^{-s-1} dr \right)w(y+x)dy.
    \end{align*}
    We swap the order of the integrals so that the last line above is equal to
    \begin{align*}
        s\left(\int_0^{\delta}  \left(\int_{-r}^r w(y+x)dy\right) r^{-s-1} dr+ \int_{\delta}^\infty\left(\int_{-\delta}^{\delta} w(y+x)dy\right)r^{-s-1} dr\right).
    \end{align*}
    By \eqref{eq: Frostman’s lemma},  the integral is bounded by the right side of \eqref{db}.
\end{proof}

\subsection{Fourier decay estimates}

We denote by $R_s(x)=|x|^{-s}$ the Riesz kernel for $s\in (0,1)$. Then the energy integral $I_s(\phi w)$ can be written as
\begin{align*}
    \int (R_s*(\phi w))(x)\overline{\phi w (x)} dx.
\end{align*}
By \cite[Theorem 3.6]{Mat15}, as a tempered distribution
\begin{align*}
    \widehat{R_s}(\xi) = \gamma(s) R_{1-s}(\xi/(2\pi)),
\end{align*}
where
\begin{align*}
    \gamma(s) = \pi^{s-1/2}\frac{\Gamma(\frac{1-s}{2})}{\Gamma(\frac{s}{2})}.
\end{align*}
If $\phi \in \cS(\BR)$ and $0<s<\alpha$, then by the Plancherel theorem,
\begin{align*}
    I_s(\phi w)=\frac{1}{2\pi}\int (\widehat{R_s*\phi w})(\xi)\overline{\widehat{\phi w} (\xi)} d\xi=\frac{\gamma(s)}{2\pi} \int R_{1-s}(\xi/(2\pi)) |\widehat{\phi w}(\xi)|^2 d\xi.
\end{align*}
By Lemma \ref{lem: energy estimate}, we conclude the following mean square estimate for $\widehat{\phi w}(\xi)$.
\begin{lemma}\label{lem: Fourier decay}
    If $0<s<\alpha$ and $\phi\in\cS(\BR)$, then
    \begin{align*}
        \frac{\gamma(s)}{(2\pi)^s}\int_{-\infty}^\infty |\widehat{\phi w}(\xi)|^2 |\xi|^{s-1} d\xi = I_s(\phi w) \lesssim_{\alpha,s} \|\phi\|_{L^2(w)}^2.
    \end{align*}
\end{lemma}

\subsection{Frequency projection}
Now we let $\lambda>0$ (which will be taken to be the spectral parameter of the Maass form) and take a parameter $\beta>0$ satisfying $1\leq \beta \leq \lambda$. Let $\eta\in C_0^\infty(\BR)$ so that $\widehat{\eta}\in C_c^\infty(\BR)$ is nonnegative even function, $\widehat{\eta}=1$ on $[-1/2,1/2]$ and $\widehat{\eta}$ vanishes outsied $(-1,1)$. Let $\eta_\beta \in C_0^\infty(\BR)$ be such that
\[
\widehat{\eta_\beta}(\xi) = \widehat{\eta}(\frac{\xi-\lambda}{\beta}) + \widehat{\eta}(\frac{\xi+\lambda}{\beta}).
\]
We define the frequency projection $\Pi_\beta$ to be the operator
\begin{align*}
    \Pi_\beta \phi (x) = \int_{-\infty}^\infty \phi(x-y) \eta_\beta(y)dy.
\end{align*}

\begin{lemma}\label{lem: decay estimate for eta beta}
    For any $N\geq1$, we have
    \[
    \eta_\beta (x) \lesssim_N \beta(1+\beta|x|)^{-N}.
    \]
\end{lemma}
\begin{proof}
    By the Fourier inversion, we have
    \begin{align*}
    \eta_\beta(x) &= \frac{1}{2\pi}\int_{-\infty}^\infty \left(\widehat{\eta}(\frac{\xi-\lambda}{\beta}) + \widehat{\eta}(\frac{\xi+\lambda}{\beta})\right)e^{ix\xi} d\xi=\frac{1}{2\pi} \beta(e^{i\lambda x}+e^{-i\lambda x})\int_{-\infty}^\infty \widehat{\eta}(\xi) e^{i\beta x\xi} d\xi,
    \end{align*}
    so the lemma follows by integrating by parts.
\end{proof}
We let $\Pi_\beta^\perp = 1-\Pi_\beta$, so that $\Pi_\beta^\perp\phi = \phi - \Pi_\beta\phi$. It is direct by definition to see the following.

\begin{lemma}\label{lem: property of freq proj Pi}
    If $\phi\in\cS(\BR)$, then we have
    \begin{enumerate}
        \item $\widehat{\Pi_\beta\phi}$ is supported inside $\pm[\lambda-\beta,\lambda+\beta]$.
        \item The support of $\widehat{\Pi^\perp_\beta\phi}$ is disjoint from $\pm[\lambda-\beta/2,\lambda+\beta/2]$.
    \end{enumerate}
\end{lemma}

\section{Maass forms on arithmetic hyperbolic surfaces}\label{section: arithmetic backgroup}

\subsection{Arithmetic hyperbolic surfaces}\label{sec: arithmetic surface and Maass forms}
We first recall the constructions of compact congruence arithmetic hyperbolic surfaces from quaternion division algebras over $\BQ$.

\subsubsection{Quaternion algebras}

Let $A=\left(\frac{a,b}{\BQ}\right)$ be a quaternion division algebra over $\BQ$, where $a,b\in\BZ$ are squarefree and $a>0$. We choose a basis $1,\omega,\Omega,\omega\Omega$ for $A$ over $\BQ$ that satisfies $\omega^2=a$, $\Omega^2=b$ and $\omega\Omega+\Omega\omega=0$. We denote the reduced norm and trace on $A$ by $\mathrm{nrd}(\alpha)=\alpha\overline{\alpha}$ and $\mathrm{trd}(\alpha)=\alpha+\overline{\alpha}$, where $\alpha\mapsto \overline{\alpha}$ is the standard involution on $A$. We let $R$ be a maximal order in $A$ (or more generally an Eichler order). For any integer $m\geq1$, we let
\[
R(m)=\{\alpha\in R\mid \mathrm{nrd}(\alpha)=m \}.
\]
The group $R(1)$ of elements of reduced norm 1 in $R$ acts on $R(m)$ by multiplication on the left, and $R(1)\backslash R(m)$ is finite (see \cite{Eic55}). We fix an embedding $\iota: A\to M_2(\BQ(\sqrt{a}))$, the 2-by-2 matrices with entries in $\BQ(\sqrt{a})$ by
\[
\iota(\alpha)=\begin{pmatrix}
    \overline{\xi}&\eta\\b\overline{\eta}&\xi
\end{pmatrix},
\]
where
\[
\alpha=x_0+x_1\omega+(x_2+x_3\omega)\Omega=\xi+\eta\Omega.
\]
Let
\( \BH = \{z\in\BC\mid \Im(z)>0 \}\)
be the upper half plane. The Lie group $\GL^+(2,\BR)$, consisting of 2-by-2 real matrices with positive determinant, acts on $\BH$ via the fractional linear transformation:
\begin{align*}
    g\cdot z=\frac{az+b}{cz+d},\qquad\text{ where }g=\begin{pmatrix}
        a&b\\c&d
    \end{pmatrix}\in \GL^+(2,\BR)\text{ and }z\in\BH.
\end{align*}
We define the lattice $\Gamma=\iota(R(1))\subset \SL(2,\BR)$, which is cocompact as we assumed $A$ to be a division algebra, and let $X=\Gamma\backslash\BH$ be the corresponding compact arithmetic hyperbolic surface.

\subsubsection{Hecke operators}\label{sec: Hecke}
We define the Hecke operators $T_n:L^2(X)\to L^2(X)$, $n\geq1$, by
\[
T_nf(z)=\frac{1}{\sqrt{n}}\sum_{\alpha\in R(1)\backslash R(n)} f(\iota(\alpha)z).
\]
There is a positive integer $q$, depending on $R$, such that for any positive integers $m$ and $n$ so that $(m,q)=(n,q)=1$, $T_n$ has the following properties:
\begin{align}
    &T_n=T_n^*,\qquad\text{ that is, }T_n\text{ is self-adjoint},\notag\\
    &T_mT_n=T_nT_m=\sum_{d|(m,n)} T_{mn/d^2}.\label{eq: Hecke relation}
\end{align}

\subsubsection{Lie groups}
We let $G=\mathrm{PSL}(2,\BR)$ and let $K=\mathrm{PSO}(2)$, which is a maximal compact subgroup of $G$. Let $A$ be the diagonal subgroup of $G$, with parameterization
\[
a(y)=\begin{pmatrix}
    e^y&0\\0&1
\end{pmatrix},
\]
and let $N=\left\{ \left(\begin{smallmatrix}
    1&*\\0&1
\end{smallmatrix}\right) \right\}$ be the unipotent subgroup. We denote the Lie algebra of $G$ by $\fg$ and equip $\fg$ with the norm
\begin{align*}
    \|\cdot\|:\begin{pmatrix}
        X_1&X_2\\X_3&-X_1
    \end{pmatrix}\mapsto \sqrt{X_1^2+X_2^2+X_3^2}.
\end{align*}
This norm defines a left-invariant metric on $G$, which we denote by $d$. We define the function $A: G\to \BR$ through the Iwasawa decomposition, that is $A(g) \in \BR$ is the unique number so that $g \in N a(A(g))K$. As $\BH = G/K$, we let $dg$ be the Haar measure on $G$ that is the extension of the usual hyperbolic volume by the measure of mass 1 on $K$.

\subsubsection{Hecke-Maass forms}

We let $\Delta$ be the Laplace-Beltrami operator on $\BH$ and $X$, which is induced from the standard hyperbolic Riemannian metric. Let $\psi\in L^2(X)$ be a Hecke-Maass form that is an eigenfunction of $\Delta$ and the operators $T_n$ with $(n,q)=1$. We let $\lambda(n)$ be the Hecke eigenvalues of $\psi$ and $\lambda$ be its spectral parameter, so that
\begin{equation}\label{def}
    \begin{aligned}
           &T_n\psi=\lambda(n)\psi,\\
    &\Delta\psi+(1/4+\lambda^2)\psi=0. 
    \end{aligned}
\end{equation}

We assume that $\| \psi \|_2=1$ with respect to the hyperbolic volume on $X$ and assume that $\lambda>0$ is large. Note that because $\Delta$ and $T_n$ with $(n,q)=1$ are self-adjoint, we may assume that $\psi$ is real-valued. For functions on $X$, we will also think of them as functions on $G$ that are left $\Gamma$-invariant and right $K$-invariant.

\subsection{Spherical functions and Harish-Chandra transforms}

In this section, we review the definitions and basic properties of spherical functions and Harish-Chandra transforms on $G=\mathrm{PSL}(2,\BR)$ and $\BH$, and we refer to \cite[Chap.\,4]{Hel84}.

For any $s\in\BC$, we let $\varphi_s$ be the spherical function on $G$ of spectral parameter $s$, which can be obtained by the integral formula:
\[
\varphi_s(g) = \int_K e^{(is+1/2)A(kg)} dk
\]
where $dk$ is the probability Haar measure on $K$.
As $\varphi_s$ is right $K$-invariant, we view $\varphi_s$ as a function on $\BH=G/K$ as well. Equivalently, $\varphi_s$ is the unique smooth left $K$-invariant function on $\BH$ satisfying $\varphi_s(i)=1$ and $\Delta\varphi_s + (1/4+s^2)\varphi_s=0$. Moreover, $\varphi_s$ is invariant under the Weyl group, that is, $\varphi_s=\varphi_{-s}$ for any $s\in\BC$.

We denote by $C_c^\infty(G\sslash K)$ the space of compactly supported smooth bi-$K$-invariant functions on $G$. Given $f\in C_c^\infty(G\sslash K)$, the Harish-Chandra transform of $f$ is given by
\begin{align*}
    \sH f(s) = \int_G f(g)\varphi_{-s}(g) dg,\qquad  s\in\BC.
\end{align*}
The Paley-Wiener theorem of Gangolli \cite{Gan71} says that the Harish-Chandra transform $\sH$ gives an isomorphism from $C_c^\infty(G\sslash K)$ onto $\mathcal{PW}_W(\BC)$, the space of Paley-Wiener type even entire functions. Moreover, the inverse Harish-Chandra transform is given by the integral formula:
\[
f(g)=\int_0^\infty \sH f(s) \varphi_s(g) d\nu_{\mathrm{PL}} (s)
\]
where $d\nu_{\mathrm{PL}}=(2\pi)^{-1}s\tanh(\pi s)ds$ is the Plancherel measure for $\BH$.

\subsection{Amplification inequality}
We choose $g_0\in G$ so that our geodesic segment $\ell$ is the image of $\{g_0a(x)\mid 0\leq x\leq1 \}$ under the projection $G\to\BH\to \Gamma\backslash\BH$. Since the weight $w(x)$ in \eqref{2.1} is 
supported in $|x|\le 2$, we
fix a function $b\in C_c^\infty([-3,3])$, which we may assume to be real-valued. Following Marshall's notation \cite[\S3]{Mar16}, if $\phi \in L^1_{\mathrm{loc}}(\BR,dx)$, we let
\begin{align*}
    \langle \psi,b\phi\rangle=\int_{-\infty}^\infty \overline{\phi(x)}\psi(g_0a(x))b(x)dx.
\end{align*}
Motivated by the problem studied in this paper, if $\phi \in L^1_{\mathrm{loc}}(\BR,w(x)dx)$, we define the weighted integral
\begin{equation*}\label{dual}
    \langle \psi,b\phi\rangle_w=\langle \psi,b\phi w\rangle=\int_{-\infty}^\infty \overline{\phi(x)}\psi(g_0a(x))b(x)w(x)dx.
\end{equation*}

We fix a real-valued function $h\in C^\infty(\BR)$ of Paley-Wiener type that is nonnegative and satisfies $h(0)=1$. Define $h_\lambda^0$ by $h_\lambda^0(s)=h(s-\lambda)+h(-s-\lambda)$, and let $k_\lambda^0$ be the $K$-bi-invariant function on $\BH$ with Harish-Chandra transform $h_\lambda^0$. The Paley–Wiener theorem implies that $k_\lambda^0$ is of
compact support that may be chosen arbitrarily small. Define $k_\lambda=k_\lambda^0*k_\lambda^0$, which has Harish-Chandra transform $(h_\lambda^0)^2$. If  $\phi \in L^1_{\mathrm{loc}}(\BR,dx)$ and $g\in G$, we let
\begin{align*}
    I(\lambda,\phi,g)=\iint_{-\infty}^\infty b(x_1)b(x_2)  \overline{\phi(x_1)}\phi(x_2)k_\lambda(a(-x_1)ga(x_2)) dx_1 dx_2.
\end{align*}

By using a partition of unity on $b$, we can assume that the supports of $b$ and $k$ are small enough so that $I(\lambda,g,\phi)$ unless $d(g,e)\leq 1$. Let $N\geq1$ be an integer, and let $\{\alpha_n\}_{n=1}^N$ be complex numbers so that $\alpha_n=0$ if $(n,q)\neq1$, both to be chosen later. Define the Hecke operator
\begin{align*}
    \cT = \sum_{n=1}^N \alpha_n T_n.
\end{align*}
We recall the amplification inequality of Marshall \cite[Proposition 3.1]{Mar16}.

\begin{proposition}\label{prop: amplification inequality}
    If $\phi\in L^1_{\mathrm{loc}}(\BR,dx)$, we have
    \begin{align*}
        |\langle \cT \psi,b\phi\rangle |^2 \leq \sum_{m,n\leq N} |\alpha_m\alpha_n| \sum_{d|(m,n)} \frac{d}{\sqrt{mn}}\sum_{\gamma\in R(mn/d^2)} |I(\lambda,\phi,g_0^{-1}\gamma g_0)|.
    \end{align*}
\end{proposition}

\section{Estimates of integrals on the geometric side}\label{section: geometric side estimates}

\subsection{Bounding $I(\lambda,\Pi_\beta^\perp(\phi w), e)$}
We take a parameter $\beta>0$ satisfying $\lambda^{\epsilon'}\leq \beta \leq \lambda$ for some $\epsilon'>0$. By Lemma \ref{lem: property of freq proj Pi}, the Fourier support of $\Pi_\beta^\perp(\phi w)$ is disjoint with $\pm [\lambda-\beta/2,\lambda+\beta/2]$. Let $b_1\in C_c^\infty(\BR)$ be a cutoff function that is equal to 1 on $[-6,6]$ and zero outside $[-7,7]$ so that
\begin{align*}
    I(\lambda,\Pi_\beta^\perp(\phi w),e)=\iint_{-\infty}^\infty \overline{b\Pi_\beta^\perp(\phi w)(x_1)}b\Pi_\beta^\perp(\phi w)(x_2)   b_1(x_2-x_1)k_\lambda(a(x_2-x_1)) dx_1 dx_2.
\end{align*}
Let $\epsilon_0>0$. We decompose $b_1(x)$ dyadicly as
\begin{align*}
    b_1(x) = b_1(x)b_1(\beta^{1-\epsilon_0} x) + \sum_{n=1}^\infty b_1(x)\left( b_1(2^{-n}\beta^{1-\epsilon_0} x)-b_1(2^{-n+1}\beta^{1-\epsilon_0} x) \right).
\end{align*}
We also have the corresponding dyadic decomposition for $I(\lambda,\Pi_\beta^\perp(\phi w),e)$:
\begin{align*}
    I(\lambda,\Pi_\beta^\perp(\phi w),e)=I_0+\sum_{n=1}^\infty I_n
\end{align*}
where
\begin{align*}
    I_0=\iint_{-\infty}^\infty \overline{b\Pi_\beta^\perp(\phi w)(x_1)}b\Pi_\beta^\perp(\phi w)(x_2)   b_1(x_2-x_1)b_1(\beta^{1-\epsilon_0} (x_2-x_1))k_\lambda(a(x_2-x_1)) dx_1 dx_2
\end{align*}
and
\begin{multline*}
    I_n=\iint_{-\infty}^\infty \overline{b\Pi_\beta^\perp(\phi w)(x_1)}b\Pi_\beta^\perp(\phi w)(x_2)   b_1(x_2-x_1)\\
    \left( b_1(2^{-n}\beta^{1-\epsilon_0}( x_2-x_1))-b_1(2^{-n+1}\beta^{1-\epsilon_0} (x_2-x_1)) \right)k_\lambda(a(x_2-x_1)) dx_1 dx_2.
\end{multline*}

\begin{lemma}    $I_0\lesssim_{\alpha}\lambda^{1/2} \beta^{-(1-\epsilon_0)(\alpha-1/2)}\|\phi\|_{L^2(w)}^2$
\end{lemma}
\begin{proof}
    We let $\delta_0$ be the delta distribution at 0 and write $\Pi_\beta^\perp(\phi w)(x)$ as
    \[
     \int_{-\infty}^\infty (\delta_0(y)-\eta_\beta(y))\phi w(x-y) dy.
    \]
    If we define $J_0(y_1,y_2) $ to be the integral
    \begin{align*}
        \iint_{-\infty}^\infty b(x_1)b(x_2)\overline{\phi w(x_1-y_1)}\phi w(x_2-y_2)b_1(x_2-x_1)b_1(\beta^{1-\epsilon_0} (x_2-x_1))k_\lambda(a(x_2-x_1)) dx_1 dx_2,
    \end{align*}
    then Fubini's theorem for distributions gives
    \begin{align}\label{eq: I0 as J0 int}
        I_0 = \iint_{-\infty}^\infty (\delta_0(y_1)-\overline{\eta_\beta(y_1)})(\delta_0(y_2)-\eta_\beta(y_2))J_0(y_1,y_2)dy_1 dy_2.
    \end{align}
    By the mean inequality, $J_0(y_1,y_2)$ is bounded by
    \begin{align*}
        &\lesssim\iint b(x_1)b(x_2)b_1(x_2-x_1)b_1(\beta^{1-\epsilon_0} (x_2-x_1))|k_\lambda(a(x_2-x_1))|\\        &\qquad\qquad\qquad\qquad\qquad\qquad\left(|\phi (x_1-y_1)|^2+|\phi (x_2-y_2)|^2 \right) w(x_1-y_1)w(x_2-y_2) dx_1 dx_2\\
         &\lesssim\iint b(x_1)b(x_2)b_1(x_2-x_1)b_1(\beta^{1-\epsilon_0} (x_2-x_1))|k_\lambda(a(x_2-x_1))|\\
        &\qquad\qquad\qquad\qquad\qquad\qquad|\phi (x_1-y_1)|^2 w(x_1-y_1)w(x_2-y_2) dx_1 dx_2\\
        &=\int\left(\int b(x_2)b_1(x_2-x_1)b_1(\beta^{1-\epsilon_0} (x_2-x_1))|k_\lambda(a(x_2-x_1))|w(x_2-y_2)dx_2\right)\\
        &\qquad\qquad\qquad\qquad\qquad\qquad b(x_1)|\phi (x_1-y_1)|^2 w(x_1-y_1) dx_1
    \end{align*}
    By \cite[Lemma 2.8]{Mar16HigerRank}, $k_\lambda(a(x_2-x_1) \lesssim \lambda (1+\lambda|x_2-x_1|)^{-1/2} \leq \lambda^{1/2}|x_2-x_1|^{-1/2}$, so
    \begin{align*}
        J_0(y_1,y_2)\lesssim \|\phi\|^2_{L^2(w)} \int_{|x_2-x_1| \leq 7 \beta^{-1+\epsilon_0}} \lambda^{1/2} |x_2-x_1|^{-1/2} w(x_2-y_2)dx_2.
    \end{align*}
    Applying Lemma \ref{lem: one variable energy estimate} to $s=1/2$ and $\delta=7 \beta^{-1+\epsilon_0}$, we obtain \[J_0(y_1,y_2)\lesssim_{\alpha}\lambda^{1/2} \beta^{-(1-\epsilon_0)(\alpha-1/2)}\|\phi\|_{L^2(w)}^2.\]
    We get the same bound for $I_0$ by combining \eqref{eq: I0 as J0 int} with Lemma \ref{lem: decay estimate for eta beta}.
\end{proof}

\begin{lemma}
    $I_n\lesssim_{\epsilon_0,\epsilon',N} 2^{-n} \lambda^{-N} \|\phi \|_{L^2(w)}^2$ for any $N\geq1$.
\end{lemma}
\begin{proof}
    We first recall the inverse Harish-Chandra transform for $k_\lambda$:
    \begin{align*}
        k_\lambda(a(x))=\int_0^\infty h_\lambda(s) \varphi_s(a(x)) d\nu_{\mathrm{PL}}(s).
    \end{align*}
    By \cite[Theorem 1.5]{Mar16HigerRank}, there are functions $f_{\pm}\in C^\infty((0,3)\times\BR_{\geq0})$ such that
    \begin{align}\label{eq: der for f pm}
        \left(\frac{\partial}{\partial x} \right)^mf_\pm(x,s)\lesssim_m x^{-m}(sx)^{-1/2}
    \end{align}
    and
    \begin{align}\label{eq: asymp exp for spherical function}
        \varphi_s(a(x))= f_+(x,s)e^{isx}+f_-(x,s)e^{-isx}+O_N((sx)^{-N})
    \end{align}
    for $x\in(0,3)$.

    After changing the variables $x_2\mapsto x_2+x_1$, applying the inverse Harish-Chandra transform in $I_n$, and applying the rapid decay of $h_\lambda$, it suffices to prove the bound 
    \begin{align}\label{eq: bound for In s}
        \iint &\overline{b\Pi_\beta^\perp(\phi w)(x_1)}b\Pi_\beta^\perp(\phi w)(x_2+x_1)   b_1(x_2)\notag\\
        &\left( b_1(2^{-n}\beta^{1-\epsilon_0}x_2)-b_1(2^{-n+1}\beta^{1-\epsilon_0} x_2) \right)\varphi_s(a(x_2)) dx_1 dx_2
        \lesssim_{\epsilon_0,\epsilon',A}2^{-n} \lambda^{-A} \|\phi \|_{L^2(w)}^2
    \end{align}
    for $|s-\lambda|\leq\beta/4$. We let
    \begin{align*}
        F(x_1)=\int b\Pi_\beta^\perp(\phi w)(x_2+x_1)   b_1(x_2)\left( b_1(2^{-n}\beta^{1-\epsilon_0}x_2)-b_1(2^{-n+1}\beta^{1-\epsilon_0} x_2) \right)\varphi_s(a(x_2)) dx_2
    \end{align*}
    and we claim that
    \begin{align}\label{eq: bound for F(x1)}
        F(x_1)\lesssim_{\epsilon_0,\epsilon',N} 2^{-n}\lambda^{-N} \|\phi \|_{L^2(w)}.
    \end{align}
    Also we notice that by Lemma \ref{lem: decay estimate for eta beta} and the Cauchy-Schwarz
    \begin{align}\label{eq: L1 bd for freq truncation}
         \int|b\Pi_\beta^\perp(\phi w)(x)| dx\leq  \int|b\phi w(x)| dx + \int\eta_\beta(y) \int|b(x)\phi w(x-y)| dx dy\lesssim\|\phi \|_{L^2(w)}.
    \end{align}
    Assuming \eqref{eq: bound for F(x1)}, the integral in \eqref{eq: bound for In s} is
    \begin{align*}
        \lesssim_{\epsilon_0,\epsilon',N} 2^{-n}\lambda^{-N} \|\phi \|_{L^2(w)} \int|b\Pi_\beta^\perp(\phi w)(x_1)| dx_1\lesssim2^{-n}\lambda^{-N} \|\phi \|_{L^2(w)}^2,
    \end{align*}
    which proves \eqref{eq: bound for In s}. It remains to prove \eqref{eq: bound for F(x1)}.  We shall estimate this integral by applying \eqref{eq: asymp exp for spherical function} to $F(x_1)$. By \eqref{eq: L1 bd for freq truncation}, the error term in \eqref{eq: asymp exp for spherical function} makes a contribution of
    \begin{align*}
        (s2^n\beta^{-1+\epsilon_0})^{-N}\int b\Pi_\beta^\perp(\phi w)(x_2+x_1)   b_1(x_2)\left( b_1(2^{-n}\beta^{1-\epsilon_0}x_2)-b_1(2^{-n+1}\beta^{1-\epsilon_0} x_2) \right)dx_2\\
        \lesssim_{\epsilon_0,N} \lambda^{-N} 2^{-Nn}\|\phi \|_{L^2(w)},
    \end{align*}
    which may be ignored. We shall only consider the integral involving $f_+$, as the other term is similar.
    We apply the Fourier inverison for $\Pi_\beta^\perp(\phi w)(x_1+x_2)$, so the term in $F(x_1)$ involving $f_+$ is equal to
    \begin{align*}
        \frac{e^{i\xi x_1}}{2\pi}\int(1-\widehat{\eta_\beta}(\xi))\widehat{\phi w}(\xi)\int b(x_2+x_1)   b_1(x_2)\left( b_1(2^{-n}\beta^{1-\epsilon_0}x_2)-b_1(2^{-n+1}\beta^{1-\epsilon_0} x_2) \right)\\
        f_+(x_2,s) e^{i(s+\xi)x_2}dx_2 d\xi\notag.
    \end{align*}
    If we replace $x_2$ with $2^n\beta^{-1+\epsilon_0}x_2$, this becomes
    \begin{align*}
        \lesssim 2^n\beta^{-1+\epsilon_0}\int(1-\widehat{\eta_\beta}(\xi))\left|\widehat{\phi w}(\xi)\right|\left|\int b(2^n\beta^{-1+\epsilon_0}x_2+x_1)   b_1(2^n\beta^{-1+\epsilon_0}x_2)\left( b_1(x_2)-b_1(2x_2) \right)\right.\\
        \left.f_+(2^n\beta^{-1+\epsilon_0}x_2,s) e^{i(s+\xi)2^n\beta^{-1+\epsilon_0}x_2}dx_2\right| d\xi.
    \end{align*}
    By \eqref{eq: der for f pm}, the $k$-th derivative of the amplitude factor is $\lesssim_k \lambda^{-1/2}2^{n/2}\beta^{-(1-\epsilon_0)/2}$. Integrating by parts $k$ times with respect to $x_2$ gives that the term in $F(x_1)$ involving $f_+$ is bounded by
    \begin{align*}
        &\lesssim_k \lambda^{-1/2}2^{-kn+n/2}\beta^{(k-1/2)(1-\epsilon_0)}\int_{|s\pm\xi|\geq\beta/4}\left|\widehat{\phi w}(\xi)\right| |s+\xi|^{-k} d\xi.
    \end{align*}
    We define $\phi_s(x)=\phi(x)e^{isx}$ so that
    \begin{align*}
        &\int_{|s\pm\xi|\geq\beta/4}\left|\widehat{\phi w}(\xi)\right| |s+\xi|^{-k} d\xi\leq\int_{|\xi|\geq\beta/4}\left|\widehat{\phi w}(\xi-s)\right| |\xi|^{-k} d\xi=\int_{|\xi|\geq\beta/4}\left|\widehat{\phi_s w}(\xi)\right| |\xi|^{-k} d\xi.
    \end{align*}
    By applying Lemma \ref{lem: Fourier decay}, we get
    \begin{align*}
        \int_{|\xi|\geq\beta/4}\left|\widehat{\phi_s w}(\xi)\right| |\xi|^{-k} d\xi&=\int_{|\xi|\geq\beta/4}\left|\widehat{\phi_s w}(\xi) \xi^{-1/4}\right| |\xi|^{-k+1/4} d\xi\\
        &\leq  \left(\int\left|\widehat{\phi_s w}(\xi) \right|^2 |\xi|^{-1/2} d\xi\right)^{1/2}\left(\int_{|\xi|\geq\beta/4}|\xi|^{-2k+1/2} d\xi\right)^{1/2}\\
        &\lesssim_{\alpha} \beta^{-k+3/4}\|\phi\|_{L^2(w)}.
    \end{align*}
    In summary, 
    \begin{align*}
        F(x_1)\lesssim_{\alpha,k} \lambda^{-1/2}2^{-kn+n/2}\beta^{(k-1/2)(1-\epsilon_0)} \beta^{-k+3/4}\|\phi\|_{L^2(w)}\lesssim_{\epsilon_0,\epsilon',N}2^{-n}\lambda^{-N}\|\phi\|_{L^2(w)},
    \end{align*}
    which proves \eqref{eq: bound for F(x1)}.
\end{proof}

By combining the above two lemmas and summing over $n$, we get:
\begin{proposition}\label{prop: bound for I(e) away from spec}
    Suppose that $\lambda^{\epsilon'}\leq \beta \leq \lambda$ for some $\epsilon'>0$ and $\phi\in\cS(\BR)$. Then we have
    \begin{align*}
        I(\lambda,\Pi_\beta^\perp(\phi w),e)\lesssim_{\alpha,\epsilon',\epsilon} \lambda^{1/2} \beta^{-(\alpha-1/2)+\epsilon} \|\phi\|_{L^2(w)}^2.
    \end{align*}
\end{proposition}

\subsection{Bounding $I(\lambda,\Pi_\beta(\phi w),g)$}
We first give a uniform bound for $I(\lambda,\Pi_\beta(\phi w),g)$.
\begin{lemma}\label{lem: uniform local bd}
    Suppose that $\phi\in\cS(\BR)$ and $\beta\leq\lambda$. We have  $I(\lambda,\Pi_\beta(\phi w),g)\lesssim_\alpha \lambda^{1/2} \|\phi\|^2_{L^2(w)} $.
\end{lemma}
\begin{proof}
    By writing $\Pi_\beta(\phi w)$ as
    \(
     \int \eta_\beta(y)\phi w(x-y) dy,
    \)
   and defining
    \begin{align*}
       J(\lambda;y_1,y_2)=\iint b(x_1)b(x_2)  \overline{\phi w(x_1-y_1)}\phi w(x_2-y_2)k_\lambda(a(-x_1)ga(x_2)) dx_1 dx_2 ,
    \end{align*}
    then we have
    \begin{align}\label{eq: I(g) as J(g) int}
        I(\lambda,\Pi_\beta(\phi w),g) = \iint\overline{\eta_\beta(y_1)}\eta_\beta(y_2)   J(\lambda;y_1,y_2)dy_1dy_2.
    \end{align}
    By the mean inequality we obtain
    \begin{align*}
        J(\lambda;y_1,y_2)
        \lesssim& \iint b(x_1)b(x_2)|k_\lambda(a(-x_1)ga(x_2))||\phi (x_2-y_2)|^2 w(x_1-y_1)w(x_2-y_2) dx_1 dx_2\\
        &+ \iint b(x_1)b(x_2)|k_\lambda(a(-x_1)ga(x_2))||\phi (x_1-y_1)|^2 w(x_1-y_1)w(x_2-y_2) dx_1 dx_2.
    \end{align*}
    We claim that
    \begin{align*}
        &\iint b(x_1)b(x_2)|k_\lambda(a(-x_1)ga(x_2))||\phi (x_2-y_2)|^2 w(x_1-y_1)w(x_2-y_2) dx_1 dx_2\\
        =&\int\left(\int  b(x_1)w(x_1-y_1)|k_\lambda(a(-x_1)ga(x_2))|dx_1 \right) b(x_2)w(x_2-y_2)|\phi (x_2-y_2)|^2dx_2\\
        \lesssim& \lambda^{1/2} \|\phi\|^2_{L^2(w)} 
    \end{align*}
    and the other term can be bounded similarly. The result will follow from the claim and Lemma \ref{lem: decay estimate for eta beta}.
    Given $g\in G$ and $x_2\in \BR$, we let $x_0 \in \BR$ be such that $a(x_0)\cdot i= e^{x_0} i\in \BH$ is the perpendicular projection of $ga(x_2)\cdot i$ onto the vertical geodesic $\BR i\subset\BH$. Therefore, in the hyperbolic triangle formed by $a(x_1)\cdot i$, $a(x_0)\cdot i$ and $ga(x_2)\cdot i$, the side length $|x_1-x_0|$ is bounded by the hypotenuse, which is the distance between $a(x_1)\cdot i$ and $g a(x_2)\cdot i$. By \cite[Lemma 2.8]{Mar16HigerRank}, we have
    \begin{align*}
        k_\lambda(a(-x_1)ga(x_2)) \lesssim \lambda(1+\lambda|x_1-x_0|)^{-1/2}\lesssim \lambda^{1/2}|x_1-x_0|^{-1/2}.
    \end{align*}
    Therefore, by the above bound and Lemma \ref{lem: energy estimate} with $s=1/2$, we have
    \begin{align*}
        \int  b(x_1)w(x_1-y_1)|k_\lambda(a(-x_1)ga(x_2))|dx_1\lesssim \lambda^{1/2}\int  b(x_1)w(x_1-y_1)|x_1-x_0|^{-1/2}dx_1\lesssim_\alpha \lambda^{1/2},
    \end{align*}
    which completes the proof.
\end{proof}

\subsubsection{Rapid decay estimate of $I(\lambda,\Pi_\beta(\phi w),g)$}

We first recall Marshall's estimate \cite[Proposition 3.7, (22)]{Mar16}:

\begin{lemma}\label{lem: Marshall prop 3.7}
    Given $0<\epsilon'<1/2$, suppose that $\lambda^{\epsilon'}\leq\beta \leq \lambda^{1-\epsilon'}$, and $\phi\in L^2(\BR)$ with $\| \phi \|_{L^2(\BR)}=1$ and $\supp(\widehat{\phi}) \subset \pm [\lambda-\beta,\lambda+\beta]$. If $\epsilon_0>0$ and $d(g,A)\geq \lambda^{-1/2+\epsilon_0}\beta^{1/2}$, we have
    \begin{align*}
        I(\lambda,\phi,g)\lesssim_{\epsilon_0,\epsilon',N} \lambda^{-N} .
    \end{align*}
\end{lemma}
\begin{rmk}
    In  \cite[Proposition 3.7]{Mar16}, it is required that $\beta\leq\lambda^{2/3}$. However, this is not necessary, and $\lambda^{2/3}$ could be replaced with $\lambda^{1-\epsilon'}$ as above (this is also mentioned in \cite[Remark p.466]{Mar16}).
\end{rmk}

To modify Marshall's estimate in terms of $\| \phi \|_{L^2(w)}$, we apply the Fourier decay estimate for $\phi w$ as follows.

\begin{proposition}\label{prop: rapid decay away from A}
    Given $0<\epsilon'<1/2$, suppose that $\lambda^{\epsilon'}\leq \beta \leq \lambda^{1-\epsilon'}$, and $\phi\in \cS(\BR)$. If $\epsilon_0>0$ and $d(g,A)\geq \lambda^{-1/2+\epsilon_0}\beta^{1/2}$, we have
    \begin{align*}
        I(\lambda,\Pi_\beta(\phi w),g)\lesssim_{\epsilon_0,\epsilon',N} \lambda^{-N} \| \phi  \|_{L^2(w)}^2.
    \end{align*}
\end{proposition}
\begin{proof}
    By the Plancherel theorem and the support condition $\supp(\widehat{\eta_\beta}) \subset \pm [\lambda-\beta,\lambda+\beta]$ we have
\begin{align*}
    \| \Pi_\beta(\phi w) \|_{L^2(\BR)}^2 &= \frac{1}{2\pi}\int_{-\infty}^\infty |\widehat{\phi w}(\xi)|^2 \left|\widehat{\eta_\beta}(\xi)\right|^2d\xi\\
    &\lesssim\int_{\pm [\lambda-\beta,\lambda+\beta]} |\xi|^{1/2} |\widehat{\phi w}(\xi)|^2 |\xi|^{-1/2} d\xi\\
    &\lesssim \lambda^{1/2}\int_{\pm [\lambda-\beta,\lambda+\beta]}|\widehat{\phi w}(\xi)|^2 |\xi|^{-1/2} d\xi\\
    &\lesssim \lambda^{1/2}\int_\BR|\widehat{\phi w}(\xi)|^2 |\xi|^{-1/2} d\xi.
\end{align*}
Since $\alpha>1/2$, applying Lemma \ref{lem: Fourier decay} with $s=1/2$, we obtain $\| \Pi_\beta(\phi w)\|_{L^2(\BR)}^2 \lesssim_\alpha \lambda^{1/2}\| \phi  \|_{L^2(w)}^2 $. Combining this estimate, we apply Lemma \ref{lem: Marshall prop 3.7} to $\Pi_\beta(\phi w) / \|\Pi_\beta(\phi w)\|_{L^2(\BR)}$ to obtain
\begin{multline*}
    I(\lambda,\Pi_\beta(\phi w),g)=\|\Pi_\beta(\phi w)\|_{L^2(\BR)}^2\, I(\lambda,\Pi_\beta(\phi w)/ \|\Pi_\beta(\phi w)\|_{L^2(\BR)},g)\\
    \lesssim_{\epsilon_0,\epsilon',N} \lambda^{-N}\|\Pi_\beta(\phi w)\|_{L^2(\BR)}^2\lesssim_\alpha \lambda^{-N+1/2}\| \phi  \|_{L^2(w)}^2,
\end{multline*}
which completes the proof.
\end{proof}

\section{Proof of Theorem~\ref{thm: main}}\label{section: proving thm by amplification}
From the estimates in the last section, we need to take a parameter $\beta>0$ satisfying $\lambda^{\epsilon'}\leq \beta \leq \lambda^{1-\epsilon'}$ for some $\epsilon'>0$.
\begin{proposition}\label{prop: L2 est away from spec}
    For any $\phi \in\cS(\BR)$, we have $\langle \psi,b\Pi_\beta^\perp(\phi w)\rangle\lesssim_{\alpha,\epsilon',\epsilon} \lambda^{1/4+\epsilon}\beta^{-\frac{\alpha-1/2}{2}}\|\phi\|_{L^2(w)}$.
\end{proposition}
\begin{proof}
    We pass to a finite-index sublattice $\Gamma'\subset\Gamma$ such that $\Gamma'\backslash\BH$ has sufficiently large injectivity radius and apply Proposition \ref{prop: amplification inequality} with $\cT$ the identity operator. Then only $\gamma=e$ contributes to the geometric side, and we obtain the result from Proposition \ref{prop: bound for I(e) away from spec}.
\end{proof}

To prove a nontrivial improvement for $\langle \psi,b\Pi_\beta(\phi w)\rangle$, we use the amplification method and need an input on the estimation of Hecke returns, which counts how many times the Hecke operators map the geodesic $g_0A$ close to itself. We define
\begin{align*}
    M(g,n,\kappa)=\#\{\eta\in R(n)\mid d(g^{-1}\eta g,e)\leq1, d(g^{-1}\eta g,A)\leq\kappa\}.
\end{align*}
The following lemma is proved by Marshall \cite[Lemma 3.3]{Mar16}.

\begin{lemma}\label{lem: Hecke return}
    If $\kappa\le1$, we have the bound
    \[
        M(g,n,\kappa)\lesssim_\epsilon (n/\kappa)^\epsilon(n\sqrt{\kappa}+1)
    \]
    uniformly in $g$.
\end{lemma}

Now we recall the construction of an amplifier, which is equivalent to the one used in \cite{IS95} and \cite{Mar16}. Let $N\geq1$ be a large positive integer to be chosen later. If $1< p\leq\sqrt{N}$ is a prime so that $(p,q)=1$, \eqref{eq: Hecke relation} implies the following relation for the Hecke eigenvalues:
\begin{align*}
    \lambda(p)^2-\lambda(p^2)=1.
\end{align*}
For such a prime $p$, the above relation implies that if $|\lambda(p)|<1/2$ then $|\lambda(p^2)|>3/4$. We set
\begin{align*}
    \begin{cases}
        \alpha_p= \frac{\lambda(p)}{|\lambda(p)|},\;\alpha_{p^2}=0,\qquad\text{ if }|\lambda(p)|\geq 1/2,\\
        \alpha_p= 0,\;\alpha_{p^2}=\frac{\lambda(p^2)}{|\lambda(p^2)|},\qquad\text{ if }|\lambda(p)|< 1/2.
    \end{cases}
\end{align*}
We set $\alpha_n=0$ for all the other $n$. It may be seen that
\begin{align}
    \sum_{n\leq N}|\alpha_n| = \sum_{n\leq N}|\alpha_n|^2=\sum_{\substack{1<p\leq\sqrt{N}\\ (p,q)=1}} 1 \lesssim N^{1/2}.\label{eq: moment for alpha}
\end{align}
Let $\cT = \sum_{n=1}^N \alpha_n T_n$ be our amplifier. Then we have $\cT\psi = \left(\sum_{n=1}^N \alpha_n\lambda(n) \right) \psi$ and the eigenvalue of the amplifier satisfies
\begin{align}\label{eq: lower bound for amplifier eigenvalue}
    \left|\sum_{n=1}^N \alpha_n\lambda(n)\right|\geq \sum_{\substack{1<p\leq\sqrt{N}\\ (p,q)=1}} \frac{1}{2}\gtrsim_\epsilon N^{1/2-\epsilon}.
\end{align}

\begin{proposition}\label{prop: L2 est near spec}
    Given $0<\epsilon'<1/2$, supppose that $\lambda^{\epsilon'}\leq\beta\leq\lambda^{1-\epsilon'}$ and $\phi \in\cS(\BR)$. Then we have $\langle \psi,b\Pi_\beta(\phi w)\rangle\lesssim_{\alpha,\epsilon',\epsilon} \lambda^{5/24+\epsilon}\beta^{1/24}\|\phi\|_{L^2(w)}$.
\end{proposition}
\begin{proof}
    The proof is essentially the same as the proof for \cite[Lemma 3.9]{Mar16}. We include the proof here for completeness.

    By Proposition \ref{prop: amplification inequality}, we have
    \begin{align}
    |\langle \cT \psi,b\Pi_\beta(\phi w)\rangle |^2 \leq \sum_{m,n\leq N} |\alpha_m\alpha_n| \sum_{d|(m,n)} \frac{d}{\sqrt{mn}}\sum_{\gamma\in R(mn/d^2)} |I(\lambda,\Pi_\beta(\phi w),g_0^{-1}\gamma g_0)|.\label{eq: amplification inequality}
    \end{align}
    Choose $\epsilon_0>0$. Proposition \ref{prop: rapid decay away from A} implies that we only need to consider the terms in \eqref{eq: amplification inequality} with $d(g_0^{-1}\gamma g_0,e)\leq1$ and $d(g_0^{-1}\gamma g_0,A)\leq\lambda^{-1/2+\epsilon_0}\beta^{1/2}$. Lemma \ref{lem: Hecke return} gives
    \[
    M(g_0,n,\lambda^{-1/2+\epsilon_0}\beta^{1/2})\lesssim_\epsilon n^\epsilon \lambda^{\epsilon_0+\epsilon}(n\lambda^{-1/4}\beta^{1/4}+1),
    \]
    and so by the uniform estimate \eqref{lem: uniform local bd} we have
    \begin{align*}
        &\sum_{m,n\leq N} |\alpha_m\alpha_n| \sum_{d|(m,n)} \frac{d}{\sqrt{mn}}\sum_{\gamma\in R(mn/d^2)}|I(\lambda,\Pi_\beta(\phi w),g_0^{-1}\gamma g_0)|\\
        \lesssim&\left( \sum_{m,n\leq N} |\alpha_m\alpha_n| \sum_{d|(m,n)} \frac{d}{\sqrt{mn}}\lambda^{1/2} M(g_0,\frac{mn}{d^2},\lambda^{-1/2+\epsilon_0}\beta^{1/2})+O_{\epsilon_0,\delta,A}(\lambda^{-A})\right)\|\phi\|_{L^2(w)}^2\\
        \lesssim&_\epsilon\left( N^\epsilon\lambda^{\epsilon_0+\epsilon}\sum_{m,n\leq N} |\alpha_m\alpha_n| \sum_{d|(m,n)} \left(\frac{\sqrt{mn}}{d}\lambda^{1/4}\beta^{1/4}+\frac{d}{\sqrt{mn}}\lambda^{1/2}\right) +O_{\epsilon_0,\delta,A}(\lambda^{-A})\right)\|\phi\|_{L^2(w)}^2.
    \end{align*}
    As in \cite[p.310]{IS95}, we have
    \begin{align*}
        \sum_{m,n\leq N}  \sum_{d|(m,n)}|\alpha_m\alpha_n|\frac{\sqrt{mn}}{d}\lesssim_\epsilon N^{1+\epsilon} \left( \sum_{n\leq N} |\alpha_n|\right)^2
    \end{align*}
    and
    \begin{align*}
         \sum_{m,n\leq N} \sum_{d|(m,n)}|\alpha_m\alpha_n|\frac{d}{\sqrt{mn}}&=\sum_{\substack{m,n\leq N\\(m,n)=1}} \sum_{ml,nl\leq N}\sum_{d|l}|\alpha_{ml}\alpha_{nl}|\frac{d}{l\sqrt{mn}}\\
         &\lesssim_\epsilon N^\epsilon\sum_{\substack{m,n\leq N}}\sum_{ml,nl\leq N}\left( \frac{|\alpha_{ml}|^2}{n} +  \frac{|\alpha_{nl}|^2}{m} \right)\\
         &\lesssim_\epsilon N^\epsilon \sum_{n\leq N} |\alpha_n|^2.
    \end{align*}
    Hence from \eqref{eq: amplification inequality} we have
    \begin{align*}
        |\langle \cT \psi,b\Pi_\beta(\phi w)\rangle |^2  \lesssim&_\epsilon\left( N^\epsilon\lambda^{\epsilon_0+\epsilon}\left(N\lambda^{1/4}\beta^{1/4}\left( \sum_{n\leq N} |\alpha_n|\right)^2+\lambda^{1/2}\sum_{n\leq N} |\alpha_n|^2\right) +O_{\epsilon_0,\delta,A}(\lambda^{-A})\right)\|\phi\|_{L^2(w)}^2.
    \end{align*}
    Combining this with \eqref{eq: moment for alpha} and \eqref{eq: lower bound for amplifier eigenvalue} and choosing $\epsilon_0$ small gives
    \begin{align*}
        N^{1-\epsilon}|\langle \psi,b\Pi_\beta(\phi w)\rangle |^2\lesssim&_\epsilon \left(N^\epsilon\lambda^{\epsilon}\left(N^2\lambda^{1/4}\beta^{1/4}+N^{1/2}\lambda^{1/2}\right)+O_{\delta,A}(\lambda^{-A})\right)\|\phi\|_{L^2(w)}^2,
    \end{align*}
    which completes the proof by choosing $N=\lambda^{1/6}\beta^{-1/6}$.
\end{proof}

Given small $\epsilon'>0$, by combining Propositions \ref{prop: L2 est away from spec} and Proposition \ref{prop: L2 est near spec}, for any $\phi\in\cS(\BR)$ and $\beta$ satisfying $\lambda^{\epsilon'}\leq\beta\leq\lambda^{1-\epsilon'}$, we have
\begin{align*}
     \langle\psi,b\phi\rangle_w = \langle\psi,b\Pi_\beta^\perp(\phi w)\rangle+\langle\psi,b\Pi_\beta(\phi w)\rangle\lesssim_{\alpha,\epsilon',\epsilon} \left(\lambda^{1/4+\epsilon}\beta^{-\frac{\alpha-1/2}{2}}+\lambda^{5/24+\epsilon}\beta^{1/24}\right)\|\phi\|_{L^2(w)}.
\end{align*}
Therefore, by choosing $\phi = b \psi|_\ell$, we obtain
\begin{align*}
    \| b\psi|_\ell \|_{L^2(w)}\lesssim_{\alpha,\epsilon',\epsilon} \lambda^{1/4+\epsilon}\beta^{-\frac{\alpha-1/2}{2}}+\lambda^{5/24+\epsilon}\beta^{1/24}.
\end{align*}
Theorem \ref{thm: main} follows from choosing 
\begin{align*}
    \beta=\lambda^{\frac{1}{1+12(\alpha-1/2)}}.
\end{align*}

\section{Proof of Theorem~\ref{thm: main1}}\label{section: proving KN bound}
In this section, we shall give the proof of Theorem~\ref{thm: main1} following ideas in the earlier work of Bourgain~\cite{bourgain2009geodesic}.

By Lemma~\ref{locallemmaa}, it suffices to estimate
$$ \left(\int_{|s|\le 2}  |e_\la(\ell(s)|^2w(s) ds\right)^{\frac12}
$$
By using \eqref{a.3}, we are reduced to bound
$$\lambda^{1/2}\left(\int_{|s|\le 2} \left|\int_Me^{i\lambda d_g(\ell(s),y)}a(\ell(s),y,\lambda) e_\la(y)dy\right|^2 w(s) ds\right)^{\frac12}
$$
By duality, let $\eta$ be a smooth function such that $\|\eta\|_{L^2(w)}=1$, we need to estimate
\begin{equation}\label{dy}
    \int \int_Me^{i\lambda d_g(\ell(s),y)}a(\ell(s),y,\lambda) e_\la(y)\eta(s)w(s) dyds
\end{equation}
We can choose Fermi coordinates geodesic $\ell$ such that $\ell(s)=(s,0)$. Now, let us decompose the $y$-range in dyadic regions according to the distance of $y$ to the geodesic $\ell$. Fix a 
Littlewood-Paley bump function satisfying
\begin{equation}\label{little}
\beta\in C^\infty_c((1/2,2)),   \, \, \, 
\beta(\tau)=1 \, \, \, \text{for } \, \tau \, \, \text{near } \, \, 1, \, \,
\text{and } \, \, \sum_{j=-\infty}^\infty \beta(2^{-j}\tau)=1, \, \, \tau>0.
\end{equation}
Denote $\beta_0(\tau)=\sum_{j=-\infty}^0\beta(|\tau|) $ and  $\beta_k(\tau)=\beta(2^{-k}|\tau|)$. Note that by \eqref{dist} we may assume  $y=(y_1,y_2)$ variable in \eqref{dy} is supported in the set $|y|\le C$ for some constant $C$ which may depend on the metric. Thus, it suffices to estimate
\begin{equation}\label{dy1}
    \int \int_Me^{i\lambda d_g(\ell(s),y)}a(\ell(s),y,\lambda) e_\la(y)\beta_k(y_2)\eta(s)w(s) dyds
\end{equation}
for $\lambda^{-\frac12}\le 2^k\le C$,   
as well as 
\begin{equation}\label{dy2}
    \int \int_Me^{i\lambda d_g(\ell(s),y)}a(\ell(s),y,\lambda) e_\la(y)\beta_0(\lambda^{\frac12}y_2)\eta(s)w(s) dyds.
\end{equation}

Write
\begin{equation}
    \Omega_k=\{y=(y_1,y_2)\mid |y_1|\le C,\,\,  2^{k-1}<|y_2|<2^{k+1}\}.
\end{equation}
We have for each fixed $k$,
\begin{multline*}
    \int \int_{\Omega_k}e^{i\lambda d_g(\gamma(x_1),y)}a(\ell(s),y,\lambda) e_\la(y)\beta_k(y_2)\eta(s)w(s) dyds\\
    \le  \left\|  \int e^{i\lambda d_g(\ell(s),y)}a(\ell(s),y,\lambda)\beta_k(y_2)\eta(s)w(s)ds\right\|_{L^2}\|e_\la\|_{L^2(\Omega_k)}
\end{multline*}
Note that 
\begin{equation}\label{6.7}
    \begin{aligned}
      &\left\|  \int e^{i\lambda d_g(\ell(s),y)}a(\ell(s),y,\lambda) \beta_k(y_2)\eta(s)w(s)ds\right\|^2_{L^2}\\
      &\qquad \le\iint   \left|\int e^{i\lambda (d_g(\ell(s),y)-d_g(\ell(s'),y))}a(\ell(s),y,\lambda)\overline{a(\ell(s'),y,\lambda)} \beta^2_k(y_2)dy\right|\\ 
&\qquad\qquad\qquad\qquad\qquad\qquad\qquad\qquad\qquad\qquad\cdot|\eta(s)w(s)\eta(s')w(s')|ds ds'.
    \end{aligned}
\end{equation}
In Fermi coordinates about the geodesic $\ell$, one can Taylor expand the distance function and show that (see e.g., Lemma A.6 in \cite{huang2025curvature})
\begin{equation}
    \partial_{y_2}(d_g(\ell(s),y)-d_g(\ell(s'),y))\sim 2^k|s-s'| \,\,\,\text{if}\,\,\, y\in \Omega_k
\end{equation}
and 
\begin{equation}
    \partial^n_{y_2}(d_g(\ell(s),y)-d_g(\ell(s'),y))\lesssim 1,\,\,\,\,\, n\ge 2.
\end{equation}
Hence, integration by parts in $y_2$ gives 
\begin{equation}\label{6.10}
\begin{aligned}
    \left|\int e^{i\lambda (d_g(\ell(s),y)-d_g(\ell(s'),y))}a(\ell(s),y,\lambda)\overline{a(\ell(s'),y,\lambda)} \beta^2_k(y_2)dy\right| \\
    \lesssim_N 2^k(1+2^{2k}\lambda|s-s'|)^{-N}
\end{aligned}
\end{equation}
Note that 
\begin{multline}
    \iint 2^k(1+2^{2k}\lambda|s-s'|)^{-N}|\eta(s)w(s)\eta(s')w(s')|ds ds'\\
    \lesssim  \iint 2^k(1+2^{2k}\lambda|s-s'|)^{-N}(\eta^2(s)+\eta^2(s'))w(s)w(s')dsds'.
\end{multline}
And by \eqref{eq: Frostman’s lemma}, it is not hard to check that 
\begin{equation}\label{6.12}
     \int 2^k(1+2^{2k}\lambda|s-s'|)^{-N}w(s)ds\lesssim 2^k \lambda^{-\alpha}2^{-2\alpha k}.
\end{equation}
Hence, \eqref{6.7}--\eqref{6.12} implies that 
\begin{equation*}
      \int \int_{\Omega_k}e^{i\lambda d_g(\ell(s),y)}a(\ell(s),y,\lambda) e_\la(y)\beta_k(y_2)\eta(s)w(s) dyds\lesssim 2^{k(\frac12-\alpha)}\lambda^{-\frac\alpha2}\|e_\la\|_{L^2(\Omega_k)}.
\end{equation*}

Let us denote 
$$\sup_{\gamma\in {\Pi}}\int_{\mathcal{T}_{\lambda^{-1/2}} (\gamma)}|e_\la|^2 dx =s_{KN}.
$$
Then since $e_\la$ is $L^2$-normalized, we must have $\lambda^{-\frac12}\lesssim s_{KN}\le 1$. Since $\Omega_k$ can be covered by $O(2^k\lambda^{1/2})$ geodesics tubes of width $\lambda^{-\frac12}$, we also have 
$$\|e_\la\|_{L^2(\Omega_k)} \lesssim \min\{1, (2^k\lambda^{1/2}s_{KN})^{\frac12}\}.
$$
Therefore,
\begin{equation*}
\begin{aligned}
        \sum_{k:\,\lambda^{-\frac12}\le 2^k\le C}\int &\int_{\Omega_k}e^{i\lambda d_g(\ell(s),y)}a(\ell(s),y,\lambda) e_\la(y)\beta_k(y_2)\eta(s)w(s) dyds\\
        &\lesssim \sum_{\lambda^{-\frac12}\le 2^k\le C}2^{k(\frac12-\alpha)}\lambda^{-\frac\alpha2}\min\{1, (2^k\lambda^{1/2}s_{KN})^{\frac12}\} \\
        &\lesssim \begin{cases}
            \lambda^{-1/4} s_{KN}^{\alpha-\frac12},\qquad &\text{if} \,\,\,1/2<\alpha<1,\\
            \lambda^{-1/4} s_{KN}^{1/2} \, \log\la ,\,\,\, &\text{if} \,\,\,\alpha=1.
        \end{cases}
\end{aligned}
\end{equation*}
To handle the remaining term in \eqref{dy2}, note that
\begin{equation*}\label{dy22}
\begin{aligned}
      &\left| \int \int_Me^{i\lambda d_g(\ell(s),y)}a(\ell(s),y,\lambda) e_\la(y)\beta_0(\lambda^{\frac12}y_2)\eta(s)w(s) dyds \right| \\
      &\qquad\lesssim   \int \int_M|a(\ell(s),y,\lambda)e_\la(y)\beta_0(\lambda^{\frac12}y_2)\eta(s)w(s)| dyds  \\
       &\qquad\lesssim  \la^{-\frac14}s^{\frac12}_{KN} \left(\int |\eta(s)w(s)|ds \right)  \lesssim \la^{-\frac14}s^{\frac12}_{KN}.
\end{aligned}
\end{equation*}
This completes the proof of Theorem~\ref{thm: main1}.

\bibliographystyle{alpha}
\bibliography{refs}

\end{document}